\documentclass{article}
\title{Geometric random graphs on circles}
\date{May 2021}
\author{Omer Angel \and Yinon Spinka}

\AtEndDocument{
  \bigskip\noindent
  \textsc{Omer Angel, Yinon Spinka} \\
  Department of Mathematics, \\
  University of British Columbia \\
  \textit{Email:} \texttt{\{angel,yinon\}@math.ubc.ca}
}

\usepackage{amsmath,amsthm,amssymb}
\usepackage{enumitem, graphicx, color}
\usepackage[margin=3cm]{geometry}

\usepackage[nameinlink]{cleveref}
  \crefname{theorem}{Theorem}{Theorems}
  \crefname{thm}{Theorem}{Theorems}
  \crefname{mainthm}{Theorem}{Theorems}
  \crefname{lemma}{Lemma}{Lemmas}
  \crefname{lem}{Lemma}{Lemmas}
  \crefname{remark}{Remark}{Remarks}
  \crefname{prop}{Proposition}{Propositions}
  \crefname{defn}{Definition}{Definitions}
  \crefname{corollary}{Corollary}{Corollaries}
  \crefname{cor}{Corollary}{Corollaries}
  \crefname{section}{Section}{Sections}
  \crefname{figure}{Figure}{Figures}

\newtheorem{thm}{Theorem}[section]
\newtheorem{lemma}[thm]{Lemma}
\newtheorem{prop}[thm]{Proposition}

\newtheorem{cor}[thm]{Corollary}

\theoremstyle{definition}
\newtheorem{remark}[thm]{Remark}

\newcommand{\cG}{\mathcal{G}}
\newcommand{\N}{\mathbb{N}}
\newcommand{\Q}{\mathbb{Q}}
\newcommand{\R}{\mathbb{R}}
\newcommand{\bbS}{\mathbb{S}}
\newcommand{\Z}{\mathbb{Z}}
 \renewcommand{\Pr}{\mathbb{P}}
\newcommand{\1}{\mathbf{1}}
\newcommand{\eps}{\varepsilon}

\DeclareMathOperator\dist{dist}

\newcommand{\iid}{i.i.d.}
\newcommand{\gec}{g.e.c.}

\begin{document}

\maketitle

\begin{abstract}
  Given a dense countable set in a metric space, the infinite random geometric graph is the random graph with the given vertex set and where any two points at distance less than 1 are connected, independently, with some fixed probability.
  It has been observed by Bonato and Janssen that in some, but not all, such settings, the resulting graph does not depend on the random choices, in the sense that it is almost surely isomorphic to a fixed graph.
  While this notion makes sense in the general context of metric spaces,  previous work has been restricted to sets in Banach spaces.
  We study the case when the underlying metric space is a circle of circumference $L$, and find a surprising dependency of behaviour on the rationality of $L$.
\end{abstract}

\begin{figure}[h]
  \centering
  \vspace{20pt}
  \includegraphics[scale=1.5]{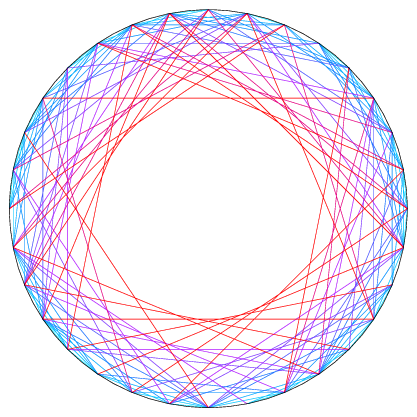}
  \caption{A random geometric graph in $\bbS_3$ with $32$ equally spaced vertices.}
  \label{fig:graph}
\end{figure}

\section{Introduction and main results}
\label{sec:introduction}

Erd{\H{o}}s and R{\'e}nyi initiated the systematic study of the random graph on $n$ vertices, in which any two vertices are connected with probability $p$, independently.
In 1964, Erd{\H{o}}s and R{\'e}nyi~\cite{erdHos1963asymmetric} showed that the infinite version of this random graph, where there are countably many vertices and any two are independently connected with probability $p$, is very different from its finite counterpart.
Specifically, there exists a fixed graph $R$ such that the infinite random graph is almost surely isomorphic to $R$.
Moreover, $R$ is the same for all $p\in(0,1)$.
Rado~\cite{rado1964universal} gave a concrete and concise description of $R$.
The graph $R$ (or, more precisely, its isomorphism type) is therefore sometimes called the \textbf{Rado graph}.
The Rado graph has several nice properties.
One such property, which in fact characterizes the graph, is that it is \textbf{existentially closed}: for any disjoint finite sets of vertices $A$ and $B$, there exists a vertex which is adjacent to every vertex in $A$ and to no vertex in $B$.
We refer the reader to~\cite{cameron1997random} for more information on the Rado graph and its properties.

The Erd\H{o}s--R\'enyi random graph, both its finite version and its infinite version, are non-geometric models -- they are random subgraphs of a complete graph.
Random geometric graphs have been studied extensively.
In these graphs, the vertices are embedded in some previously defined metric space $(X,d)$, and the probability of a connection depends on the distance between the vertices.
If the set of vertices is locally finite, the structure of the resulting graph can be expected to mirror that of the underlying metric space $X$.
However, if the set is dense in $X$, a very different story unfolds.
Bonato and Janssen~\cite{bonato2011infinite} initiated the study of random geometric graphs in which any two points of a countable metric space that are at distance less than one from each other are independently connected with probability $p$.
They introduced a property of these graphs called \textbf{geometric existentially closed} (\gec), analogous to the existentially closed property of the Rado graph.
A graph $G$ whose vertex set is a metric space $(V,d)$ is said to satisfy \gec\ if, for any vertex $s \in V$, any disjoint finite sets of vertices $A,B \subset V$ which are contained in the open unit-ball around $s$, and any $\epsilon>0$, there exists a vertex $v \in V \setminus (A \cup B)$ which is adjacent to every vertex in $A$, is not adjacent to any vertex in $B$, and satisfies that $d(v,s) < \epsilon$.
They then showed that, for any countable metric space in which every point is an accumulation point, the corresponding geometric random graph almost surely satisfies the \gec\ property.

A countable metric space is said to be \textbf{Rado} if the infinite random geometric graph has a unique isomorphism type, i.e., if two independent samples of the geometric random graph, possibly with distinct $p$, are almost surely isomorphic.
Such a space is called \textbf{strongly non-Rado} if two such samples are almost surely non-isomorphic.
When referring to these terms in the context of a countable subset of a metric space, we are actually referring to the metric space induced on that set.
Thus, if $S$ is a countable subset of a metric space $(V,d)$, then we say that $S$ is Rado (strongly non-Rado) if the metric space induced on $S$, namely $(S,d|_{S \times S})$, is Rado (strongly non-Rado).
We informally say that a metric space has the \textbf{Rado property} if a typical (e.g., generic or random) dense countable subset of it is a Rado set.
To make this precise, if the metric space has some probability measure, one can consider the set $S$ given by an infinite \iid\ sequence of samples from the measure.
The basic question then arises: which metric spaces have the Rado property?
For example, if the space has diameter less than $1$, then any two points are connected with probability $p$ and the geometric random graph is nothing but the original Rado graph.

In the case of the metric space $(\R,|\cdot|)$, Bonato and Janssen~\cite{bonato2011infinite} prove that the Rado property holds: there exists a fixed graph, denoted $GR(\R)$, such that for a generic dense countable subset of $\R$ (more precisely, for any dense countable set having no two points at an integer distance apart) and for any $p$, the random graph is almost surely isomorphic to $GR(\R)$.
They also extend this to the case of the metric space $(\R^d,\ell_\infty)$.
Here too, there is a fixed graph $GR(\R^d)$ for each $d \ge 1$.
The graphs $R, GR(\R), GR(\R^2), \dots$ are all non-isomorphic to one another.
In contrast, they show that this is not true for the Euclidean metric space $(\R^d,\ell_2)$, where every dense countable set is strongly non-Rado \cite{bonato2011infinite}, nor for the hexagonal norm on $\R^2$, where a randomly chosen dense countable set is almost surely strongly non-Rado \cite{bonato2012infinite}.
They later showed that many normed spaces fail to have the Rado property \cite{bonato2013properties}, including $(\R^d,\ell_p)$ for any $1<p<\infty$ (and also for $p=1$ when $d \ge 3$).
Balister, Bollob{\'a}s, Gunderson, Leader and Walters \cite{balister2018random} subsequently showed that $(\R^d,\ell_\infty)$ are the unique finite-dimensional normed spaces for which the Rado property holds.
In fact, in any other normed space, a generic dense countable subset is strongly non-Rado.

Certain infinite-dimensional normed spaces (all of which are Banach spaces) have also been considered.
Bonato, Janssen and Quas \cite{bonato2019geometric} studied the space $\mathfrak{c}$ of real convergent sequences with the sup norm and the subspace $\mathfrak{c}_0$ of sequences converging to zero, and showed that both have the Rado property.
They also showed that Banach spaces can be recovered from the random graph in the sense that, if two Banach spaces yield isomorphic random graphs, then the two spaces are isometrically isomorphic.
In a subsequent paper \cite{bonato2018geometric}, the same authors considered the space $C[0,1]$ of continuous functions with sup norm, and proved that the Rado property holds for certain subspaces, including the spaces of piece-wise linear paths, polynomials, and Brownian motion paths.

Though the notion of an infinite random geometric graph is defined in the general context of (countable) metric spaces, we are unaware of any previous works which deal with metric spaces other than normed spaces. One of the goals of this paper is to investigate the random graph in such spaces, and we do so through the example of the cycle.

\medbreak

Let $L>0$ and consider $\bbS_L := \R / L\Z$, the circle of circumference $L$ with its intrinsic metric (so that, for example, the diameter of the metric space is $L/2$).
Let $S$ be a dense countable subset of $\bbS_L$.
Let $G_{L,S}$ be the unit-distance graph on $S$, i.e., the graph whose vertex set is $S$ and whose edge set consists of all pairs of points in $S$ whose distance in $\bbS_L$ is less than 1.
Given $p \in (0,1)$, let $G_{L,S,p}$ be a random subgraph of $G_{L,S}$ obtained by retaining each edge of $G_{L,S}$ with probability $p$, independently for different edges.
See \cref{fig:graph} for an example with a finite set $S$.

As usual, we say that two graphs $G$ and $G'$ are isomorphic if there exists a bijection $\varphi$ from the vertex set of $G$ to that of $G'$ such that $\varphi(u)$ and $\varphi(v)$ are adjacent in $G'$ if and only if $u$ and $v$ are adjacent in $G$. In this case, we write $G \cong G'$.
If $L\le 2$, then $G_{L,S,p}$ is easily seen to be isomorphic to the Rado graph $R$. Thus, we henceforth always assume that $L \ge 2$.

Our first result is concerned with distinguishing between the different metric spaces, showing that different values of $L$ produce non-isomorphic graphs, so that one can recover the length $L$ of the circle from (the isomorphism type of) the random graph. When $L=\infty$, it is natural to interpret $\bbS_\infty$ as the metric space $(\R,|\cdot|)$.

\begin{thm}\label{thm:main0}
  For any $L \in [2,\infty]$, any dense countable $S\subset\bbS_L$ and any $p\in(0,1)$, 
  the cycle length $L$ can be recovered almost surely as a measurable function of the graph $G_{L,S,p}$ which is invariant to graph isomorphisms.
\end{thm}

\begin{remark}
  There is a minor delicate issue with the definition of recoverability of $L$.
  For a graph on some countable set of vertices, which may be enumerated by $\N$, we have the usual product $\sigma$-algebra generated by the presence of edges.
  The claim is that there exists a measurable function $f$ from the set of graphs on vertex set $\N$ to $\R_+$ such that, for any $L \in [2,\infty]$, any $p\in(0,1)$, any dense countable $S \subset \bbS_L$ and any enumeration of $S$, we almost surely have $f(G_{L,S,p}) = L$.
  Moreover, $f$ is invariant to relabeling the vertices of the graph.
  Crucially, this invariance is complete and not just probabilistic.
  That is, $f(G)$ is invariant to any permutation of the vertex labels of $G$, even if the permutation is allowed to depend on which edges are present in $G$.
  To clarify the strength of this property, consider \iid\ sequences $(X_i)$ of Bernoulli random variables.
  The expectation $p$ can be recovered almost surely from the law of large numbers as $\lim \frac1n\sum_{i\le n} X_i$.
  However, this function is not invariant to arbitrary permutations of the sequence if the permutation is allowed to depend on the sequence.
  The reason our function is strongly invariant to relabeling is that, for any $\ell$, the set of graphs with $f(G)\le \ell$ is described as those graphs satisfying a certain second-order proposition, which does not involve the labels.
\end{remark}

Our second and main result is concerned with self-distinguishability, showing that $\bbS_L$ has the Rado property if and only if $L$ is rational.
For rational $L$, we say that a set $S \subset \bbS_L$ is \textbf{integer-distance-free} if no two points in $S$ are at a distance which is a multiple of $\frac 1m$, where $L=\frac \ell m$ is irreducible.
This terminology may seem mysterious, so we remark that a set $S$ is integer-distance-free if and only if, starting at any point of $S$ and moving an integer distance along the cycle, possibly winding multiple times around the cycle, one can never terminate at another point of $S$.

\begin{thm}\label{thm:main}
  Let $L \ge 2$, let $S,S'$ be dense countable subsets of $\bbS_L$ and let $p,p' \in (0,1)$.
  Let $G=G_{L,S,p}$ and $G'=G_{L,S',p'}$ be independent. Then, almost surely,
  \begin{enumerate}
  \item $G \not\cong G'$ if $L \notin \Q$.
  \item $G \cong G'$ if $L \in \Q$ and $S$ and $S'$ are integer-distance-free.
  \end{enumerate}
\end{thm}

The theorem implies that, for rational $L$, a generic dense countable set is a Rado set, whereas, for irrational $L$, every such set is strongly non-Rado.
In the rational case, there exist dense countable sets which are non-Rado (see Remark~\ref{rem:non-rado}).
In the irrational case, we can show more -- namely, that up to isometries of $\bbS_L$, one may recover $S$ from (the isomorphism type of) the random graph.

\begin{thm}\label{thm:recover_S}
Let $L>2$ be irrational, let $S,S'$ be dense countable subsets of $\bbS_L$ and let $p,p' \in (0,1)$. Suppose that $G=G_{L,S,p}$ and $G'=G_{L,S',p'}$ can be coupled so that $G \cong G'$ with positive probability. Then $S$ and $S'$ differ by an isometry of $\bbS_L$.
\end{thm}

\section{Definitions and notation}

We view elements of $\bbS_L$ as real numbers modulo $L$, and we sometimes identify $\bbS_L$ with the interval $[0,L)$.
An \textbf{open arc} is the image of an open interval in $\R$ modulo $L$.
For $a,b \in \bbS_L$, we write $(a,b)_{\bbS_L}$ for the positive/anti-clockwise open arc from $a$ to $b$, i.e., for $(a,b)$ when $0 \le a \le b < L$ and for $(a,b+L)$ modulo $L$ when $0 \le b < a < L$. Thus, $(a,b)_{\bbS_L}$ and $(b,a)_{\bbS_L}$ partition $\bbS_L \setminus \{a,b\}$.
The length of an open arc $(a,b)_{\bbS_L}$ is $b-a$ if $b \ge a$ and is $b-a+L$ if $a>b$. When $a$ and $b$ are real numbers (not necessarily in $\bbS_L$) with $\|a-b\|<L$, we may unambiguously define $(a,b)_{\bbS_L}$ by interpreting $a$ and $b$ modulo $L$. With a slight abuse of notation, we simply write $(a,b)$ for $(a,b)_{\bbS_L}$.
Closed arcs and half-open-half-closed arcs are similarly defined.
The distance between two points $u,v \in \bbS_L$, denoted $\|u-v\|$, is the length of the shorter of the two arcs between $u$ to $v$ and can be written as
\[ \|u-v\| = \min \{ |u-v|, L - |u-v| \} .\]

Given a graph $G$, we write $N(v)$ for the neighbourhood of a vertex $v$ in $G$, and we write $\dist_G(u,v)$ for the graph-distance between vertices $u$ and $v$ in $G$. The length of a path in $G$ is the number of edges in the path, so that $\dist_G(u,v)$ is the length of a shortest path between $u$ and $v$.

A graph $G$ whose vertex set is a subset $S$ of $\bbS_L$ is called \textbf{\gec} (geometrically existentially closed) if, for any vertex $s \in S$, any disjoint finite sets $A,B \subset S$ which are contained in $(s-1,s+1)$, and any $\epsilon>0$, there exists a vertex $v \in S \setminus (A \cup B)$ which is adjacent to every vertex in $A$, is not adjacent to any vertex in $B$, and satisfies that $\|v-s\| < \epsilon$.
The graph $G$ is said to have \textbf{unit threshold} if any two adjacent vertices $u,v \in S$ satisfy that $\|u-v\|<1$.
The notions of \gec\ and unit threshold exist in the general context of a graph whose vertex set is a metric space.

For a dense countable subset $S$ of $\bbS_L$, let $\cG_{L,S}$ denote the collection of all \gec\ graphs on $S$ having unit threshold. Let $\cG_L$ denote the union of $\cG_{L,S}$ over all such $S$. As it turns out, once we establish the following simple property that $G_{L,S,p}$ belongs to $\cG_{L,S}$, almost all our arguments become completely deterministic.

\begin{lemma}\label{lem:gec}
Let $L \in [2,\infty]$, let $S$ be a dense countable subset of $\bbS_L$ and let $p \in (0,1)$. Then, almost surely, $G_{L,S,p} \in \cG_{L,S}$.
\end{lemma}
\begin{proof}
It is trivial from the definition that $G_{L,S,p}$ almost surely has unit threshold.
Fix a vertex $s \in S$, disjoint finite sets $A,B \subset S$ which are contained in the open unit-ball around $s$, and a rational $0<\epsilon< 1-\max_{u \in A \cup B} \|u-s\|$. Since $S$ is countable, it suffices to show that, almost surely, there exists a vertex $v \notin A \cup B$ which is adjacent to every vertex in $A$, is not adjacent to any vertex in $B$, and satisfies that $\|v-s\| < \epsilon$. Since the open ball of radius $\epsilon$ around $s$ contains infinitely many points $v$ which are not in $A \cup B$, and since the sets of edges incident to these $v$ are independent, it suffices to show that each such $v$ has a positive (fixed) probability to be adjacent to all of $A$ and to none of $B$. Indeed, since any such $v$ has $\|u-v\|<1$ for all $u \in A \cup B$, this event has probability $p^{|A|}(1-p)^{|B|}$.
\end{proof}

\section{Distinguishing graphs arising from different $L$}

In this section, we prove \cref{thm:main0}.
Our strategy is to introduce a graph-theoretic quantity which allows to differentiate between graphs arising from different $L$.

For a graph $G$, define $\lambda(G)$ to be the supremum over $\lambda \ge 0$ such that for every finite set of vertices $U \subset G$, there exists a vertex in $G$ having at least $\lambda |U|$ neighbours in $U$.
Thus,
\[ \lambda(G) := \inf_{\substack{U \subset G\\0<|U|<\infty}} \sup_{v \in G}\, \frac{|N(v) \cap U|}{|U|} .\]

Consider a graph $G \in \cG_{L,S}$.
It is easy to check that $\lambda(G)=1$ if $L \le 2$, since $G$ is just the Rado graph.
If $L=\infty$, then $\lambda(G)=0$, since $S$ contains arbitrarily large finite sets $U$ such that all vertices of $U$ are at distance more than $2$ from each other, so that $|N(v)\cap U| \leq 1$ for all $v$.
In fact, as we now show, $\lambda(G)$ depends on $G$ only through $L$, and moreover, is equal to $2/L$.
Theorem~\ref{thm:main0} is an immediate consequence of Lemma~\ref{lem:gec} and the following.

\begin{prop}
  Let $L \in [2,\infty]$ and $G \in \cG_L$. Then $\lambda(G) = 2/L$.
\end{prop}

\begin{proof}
Let $L \in [2,\infty]$ and $G \in \cG_L$.
By definition of $\cG_L$, the vertex set of $G$ is a dense countable subset $S$ of $\bbS_L$.
For a finite $U \subset S$ and an arc $A \subset \bbS_L$, we call $|U \cap A|/|U|$ the \emph{density of $U$ in $A$}.

We begin by proving the upper bound $\lambda(G) \le 2/L$. Since $G$ has unit threshold, for any $v$, $N(v)$ is contained in an arc of length $2$.
Thus it suffices to exhibit, for any $\epsilon>0$, a finite set $U \subset S$ whose density is no more than $2/L + \epsilon$ in any arc of length 2.
Any set that is close to evenly distributed on the cycle will do.
For completeness, here is one construction:
Let $n$ be a large integer and consider the set $V$ consisting of points $v_0,\dots,v_{n-1} \in \bbS_L$, where $v_i := iL/n$.
Since $S$ is dense, there exist a finite set $U$ consisting of points $u_0,\dots,u_{n-1} \in S$ such that $\|u_i-v_i\|<1/n$. It is straightforward to verify that, for any arc $A$ of length $r$, we have that $|U \cap A| \le |V \cap A| + 2 \le \lfloor rn/L \rfloor + 3$. Thus, $U$ has density at most $2/L + 3/n$ in any arc of length 2.

We now turn to the lower bound $\lambda(G) \ge 2/L$.
To show this, we show that the situation described above is essentially the worst case.
Precisely, given a finite $U \subset S$, we claim that there exists an open arc $A$ of length $2$ in which $U$ has density at least $2/L$.
This is easy to verify, since if $x$ is uniform in $\bbS_L$, then the expected number of points of $U$ in the arc $(x-1,x+1)$ is $(2/L)|U|$, so for some $x$ it is at least that large.
Since $S$ is dense, and $G$ is \gec, the arc $A$ contains a vertex $v\in S$ which is adjacent to all vertices of $U$ in $A$.
This proves the lower bound $\lambda(G) \ge 2/L$.
\end{proof}

\section{Recovering distances and non-isomorphism for irrational $L$}

In this section, we prove part (1) of \cref{thm:main}, namely that for irrational $L>2$, the independent graphs $G_{L,S,p}$ and $G_{L,S',p'}$ are almost surely non-isomorphic.
The key step is to show that by looking at the graph-structure of $G$ alone, it is possible to determine the distance in $\bbS_L$ between any two vertices.

Throughout this section, we fix $L \in [2,\infty]$ and assume $G \in \cG_{L,S}$ for some dense countable $S \subset \bbS_L$.
It is easy to check that, for any two vertices $u,v \in S$ and any integer $k \ge 2$, we have
\[ \|u-v\| < k \qquad \iff \qquad \dist_G(u,v) \le k .\]
Indeed, if $\|u-v\|<k$, then for some $\eps>0$, there is a path $u=x_0,x_1,\dots,x_k=v$ with $\|x_i-x_{i-1}\| < 1-\eps$.
By \gec\ there is a perturbation $x'_i$ of $x_i$ for $i=1,\dots,k-1$ so that $(x'_i,x'_{i-1})$ are edges of $G$ for all $i=1,\dots,k$.
Conversely, the unit-threshold property shows that a path of length at most $k$ in $G$ from $u$ to $v$ implies that the cycle distance is less than $k$.
Note that for $k=1$ this equivalence fails, since $\{u,v\}$ may or may not be an edge of $G$. See \cite[Theorem~2.4]{bonato2011infinite} for a similar statement.

We may rewrite the above as
\begin{equation}\label{eq:floor-dist}
  \dist_G(u,v) =
  \begin{cases}
    \lfloor \|u-v\| \rfloor +1 &\text{if } \|u-v\|\ge 1 \\
    \text{1 or 2} &\text{if } \|u-v\| < 1
  \end{cases}.
\end{equation}
Thus, distances in $G$ are predetermined for points of $S$ which are at distance at least 1 in $\bbS_L$.
However, we are more interested in the other direction of implication: forgetting that the vertices of $G$ are labelled by elements of $\bbS_L$ and looking only at the graph structure of $G$ in relation to $(u,v)$, one may recover $\lfloor \|u-v\| \rfloor$, unless it is 0 or 1.

To formalize these types of ideas, we require some definitions.
A graph with $k$ distinguished vertices is a graph $G$, together with an ordered $k$-tuple $(v_1,\dots,v_k)$ of distinct vertices of $G$.
Let $\cG_{L,\bullet,\bullet}$ denote the collection of all graphs in $\cG_L$ with two distinguished vertices.
Let $\pi$ denote the projection from $\cG_{L,\bullet,\bullet}$ to the class $\cG_{\bullet,\bullet}$ of  isomorphism classes of graphs with two distinguished vertices.
The above may be restated as saying that the function $(G,u,v) \mapsto \lfloor \|u-v\| \rfloor \1_{\{\|u-v\|\ge2\}}$ from $\cG_{L,\bullet,\bullet}$ to $\Z$ can be written as a composition of a function from $\cG_{\bullet,\bullet}$ to $\Z$ with $\pi$.
Indeed, \eqref{eq:floor-dist} gives a concrete such description, since the graph-distance between the distinguished vertices is invariant to isomorphisms of the graph.
In this case, we say that $\lfloor \|u-v\| \rfloor \1_{\{\|u-v\|\ge2\}}$ can be recovered from the graph structure of $(G,u,v)$.

More generally, we say that a function $f \colon \cG_{L,\bullet,\bullet} \to \Omega$ \textbf{can be recovered} from the graph structure of $(G,u,v)$ if $f=F \circ \pi$ for some $F \colon \cG_{\bullet,\bullet} \to \Omega$.
For brevity, we say that $f(G,u,v)$ can be recovered from $\pi(G,u,v)$.
We extend these definitions to graphs with $k$ distinguished points, writing $\pi$ also for the projection from $\cG_{L,\bullet,\dots,\bullet}$ to the corresponding set of isomorphism classes of graphs with $k$ distinguished vertices.
We shall also talk about sets of vertices being recoverable from the graph.
For example, \cref{lem:recover-arc} below says that the set of vertices in the shorter arc between $u$ and $v$ along the cycle is recoverable.
Formally, a set $A=A(G,u,v)$ of vertices of $G$ can be recovered from $\pi(G,u,v)$, if the function $\1_{\{x\in A\}}$ can be recovered from $\pi(G,u,v,x)$ for any $x \in G$.

The main ingredient in this section is the following proposition, which shows that we can recover plenty of information on the distances in $\bbS_L$ from the graph structure (for both rational and irrational $L$).
Given this, part (1) of \cref{thm:main} is easily deduced.

\begin{prop}\label{prop:recover-distances}
  Let $L>2$, let $G \in \cG_L$ and let $u,v \in G$ be adjacent.
  Then the sequence
  \[ (\lfloor \|u-v\| + kL \rfloor)_{k \ge 1} \]
  can be recovered from $\pi(G,u,v)$.
\end{prop}

The values $\|u-v\| + kL$ can be thought of as the distance from $u$ to $v$, moving $k$ additional times around the cycle instead of directly.
The assumption that $u,v$ are adjacent in $G$ can be removed from this proposition, but it simplifies the proof and does not significantly impact the application for the proof of \cref{thm:main}.
In the case of irrational $L$, this gives the following, stronger corollary, which immediately implies the more general result.

\begin{cor}\label{cor:recover-distances}
  Let $L>2$ be irrational, let $G \in \cG_L$ and let $u,v \in G$.
  Then $\|u-v\|$ can be recovered from $\pi(G,u,v)$.
\end{cor}

\begin{proof}
  Consider first the case when $u$ and $v$ are adjacent in $G$, so that \cref{prop:recover-distances} applies.
  It suffices to see that the mapping $x \mapsto (\lfloor x+kL \rfloor)_{k \ge 1}$ is injective on $[0,L)$.
  It is a well known fact that for irrational $L$ the fractional parts $(kL-\lfloor kL \rfloor)$ are dense in $[0,1]$.
  Let $0 \le x<y <L$.
  Since the fractional parts are dense, it follows that for some $k$ we have $\lfloor x+kL \rfloor \neq \lfloor y+kL \rfloor$, and the sequences differ.

  For any path in $G$, we can therefore recover from $G$ the total length of the edges along $\bbS_L$.
  If $u$ and $v$ are not adjacent, then there is a path in $G$ from $u$ to $v$ which moves around $\bbS_L$ in the shorter direction without backtracking.
  Since we can recover the cycle distance in each edge of the path, the sum is the distance from $u$ to $v$.
  Any other path in $G$ must give a larger sum.
  Thus we can recover $\|u-v\|$ as the minimal sum of cycle distances along paths from $u$ to $v$ in the graph.
\end{proof}

Since the cycle distance between any two vertices can be recovered from the graph, we have the following.

\begin{cor}\label{cor:recover-all-distances}
  Let $L>2$ be irrational, let $S,S'$ be dense countable subsets of $\bbS_L$ and $G \in \cG_{L,S}, G' \in \cG_{L,S'}$.
  If $f\colon S \to S'$ is a graph isomorphism between $G$ and $G'$, then it is an isometry between $S$ and $S'$.
\end{cor}

\cref{cor:recover-all-distances} immediately implies \cref{thm:recover_S}.
We now prove part (1) of \cref{thm:main}.

\begin{proof}[Proof of \cref{thm:main}(1)]
  Let $G=G_{L,S,p}$ and $G'=G_{L,S',p'}$ be independent, as in the statement of the theorem. By \cref{lem:gec}, we have that $G \in \cG_{L,S}$ and $G' \in \cG_{L,S'}$ almost surely.
  Consider a bijection $f \colon S \to S'$.
  By \cref{cor:recover-all-distances}, if $f$ is not an isometry between $S$ and $S'$, then it is not an isomorphism between $G$ and $G'$.
  Thus it suffices to consider isometries $f$.
  There are at most countably many isometries between $S$ and $S'$ (for an arbitrary $v_0\in S$, there are at most two isometries for any given choice of $f(v_0)$).
  Since any fixed isometry $f$ is almost surely not an isomorphism between $G$ and $G'$,
  we conclude that there almost surely does not exist an isomorphism between $G$ and $G'$.
\end{proof}

\subsection{Proof of Proposition~\ref{prop:recover-distances}}
The overall strategy for the proof of \cref{prop:recover-distances} is as follows: we define a graph-theoretic notion of a cyclic ordering of vertices.
This notion, though defined completely in terms of the graph $G$, will be such that it guarantees that the corresponding points in $\bbS_L$ are cyclically ordered as well.
This will allow to define another graph-theoretic notion of a uni-directional path in $G$, which will correspond to a path in $\bbS_L$ that winds around the circle in a fixed direction.
We then show that any uni-directional path in $G$ has a well-defined (again, in terms of the graph) winding number which counts the number of times the corresponding path in $\bbS_L$ winds around the circle.
Finally, using this we deduce that from the graph $G$, for any two adjacent vertices $u,v \in G$, we may recover the sequence $(\lfloor \|u-v\| + kL \rfloor)_{k \ge 1}$, which is \cref{prop:recover-distances}.

Fix $L>2$, a dense countable subset $S \subset \bbS_L$ and a graph $G \in \cG_{L,S}$.
For $x,y \in S$ having $\|x-y\|<1$, let $A_{x,y}$ denote the set of points of $S$ in the shorter arc of $[x,y]$ and $[y,x]$.
It is convenient to include the endpoints $x,y$ in the arc.

The starting point of our argument is the following lemma which shows that the shortest arc between two adjacent vertices can in fact be described as a graph-theoretic property.
Its proof is postponed to the end of the section.

\begin{lemma}\label{lem:recover-arc}
  Let $a,b \in G$ be adjacent.
  Then $A_{a,b}$ can be recovered from $\pi(G,a,b)$.
\end{lemma}

We say that a triplet $(a,b,c)$ of distinct points in $G$ is \textbf{cyclically ordered in $\bbS_L$} if $A_{a,b} \cap A_{b,c} = \{b\}$.
This means that when moving from $a$ to $b$ to $c$ along the cycle (always via the shorter arc), the direction of movement is maintained.
We say that a path $p=(v_0,\dots,v_n)$ in $G$ consisting of distinct vertices is a \textbf{uni-directional} if the triplet $(v_{i-1},v_i,v_{i+1})$ is cyclically ordered in $\bbS_L$ for each $1 \le i \le n-1$.
Thus, by going along the shorter arc between consecutive vertices, a uni-directional path in $G$ may be thought to correspond to a continuous path in $\bbS_L$ that always winds around the circle in a single direction.
In light of \cref{lem:recover-arc}, this property can be determined from the graph structure of $G$, so that we may talk about the path being uni-directional in $G$.
The \textbf{winding number} of a uni-directional path $p$ is defined as the number of complete revolutions its continuous counterpart makes around the cycle --
if its total cycle-length is $\ell$, this is $\lfloor \ell/L\rfloor$.
The winding number can also be calculated as the number of indices $1 \le i \le n-1$ such that $v_0 \in A_{v_i,v_{i+1}}$.
Consequently, the winding number of a uni-directional path $p$ can be recovered from $\pi(G,v_0,\dots,v_n)$.

It will also be useful to be able to identify the direction in which a uni-directional path winds around the circle;
by this we do not mean the absolute direction (clockwise/anticlockwise), but rather whether it goes from the start point to the end point by starting through the short/long arc between them.
This can be done in one of several ways.
We choose here a simple definition, which comes at the cost of requiring the start and end points to be at distance less than~1.
For $u,v \in S$, a \textbf{good path} from $u$ to $v$ is a uni-directional path $p=(x_0,x_1,\dots,x_n)$ in $G$ such that $x_0=u$, $x_n=v$ and $v\in A_{u,x_1}$.
Thus, a good path is required to go towards $v$ in the shorter direction, and overshoot $v$ in its first step.
In particular, its winding number is at least 1.
Of course, this is only possible if $\|u-v\|<1$.
The following shows that good paths exist.

\begin{lemma}\label{lem:good_paths_exist}
  Let $k \ge 1$ and $u,v \in G$ be such that $\|u-v\|<1$.
  Then $G$ contains a good path from $u$ to $v$ with winding number $k$ and length $n = \lfloor \|u-v\| + kL \rfloor + 1$.
  Moreover, there is no good path from $u$ to $v$ with winding number $k$ and length less than $n$.
\end{lemma}

\begin{proof}
  For concreteness, we assume that the short arc from $u$ to $v$ is in the positive direction, so that $v = u + \|u-v\|$.
  Set $\ell:=\|u-v\|$, so that we seek a path of length $n = \lfloor \ell+kL \rfloor + 1$.
  We start by specifying approximate locations for the points of the path.
  These approximate locations, denoted $x_i$, are points of the circle and need not be vertices of the graph.
  Let $x_1 = u+t$ for some $t\in(\ell,1)$, so that $v\in A_{u,x_1}$ and $\|x_1-u\| < 1$.
  The total cycle-length of the path will be $\ell + kL$, and the remaining points $(x_i)_{i=2}^{n-1}$ will be equally spaced with gap
  \[ \Delta = \frac{\ell + kL - t}{n-1}. \]
  Thus, the approximate points are $x_i = u+t+\Delta(i-1)$ for $1 \le i \le n-1$.
  Note that for $t$ close enough to $1$, we have $\Delta<1$, since $\ell+kL < n$.

  Fix $\eps>0$ such that $\max\{t,\Delta\} < 1-2\eps$,
  and set $U_0 := \{u\}$, $U_n := \{v\}$ and $U_i := S \cap (x_i-\eps, x_i+\eps)$ for $1 \le i \le n-1$.
  This choice guarantees that any point in $U_i$ is at distance less than 1 from any point in $U_{i-1}$, and the shorter arc between them is positively oriented.
  Since $G$ is \gec, there exists $u_1 \in U_1$ such that $u_1$ is adjacent to $u_0:=u$ in $G$.
  Continuing by induction, we see that there exists a sequence $(u_i)_{0 \le i \le n-2}$ such that $u_i \in U_i \setminus \{v,u_0,\dots,u_{i-1}\}$ and $u_i$ is adjacent to $u_{i-1}$ in $G$.
  Finally, there exists $u_{n-1} \in U_{n-1} \setminus \{v,u_0,\dots,u_{n-2}\}$ which is adjacent to both $u_{n-2}$ and $u_n := v$.
  By construction, $(u_0,\dots,u_n)$ is a good path in $G$ from $u$ to $v$ of length $n$ and winding number $k$.
  
  Finally, there can be no uni-directional path (good or otherwise) from $u$ to $v$ with winding number $k$ and length at most $\lfloor \ell+kL \rfloor$, since the total length of the arcs in such a path is smaller than $\ell+kL$.
\end{proof}

We are now ready to prove Proposition~\ref{prop:recover-distances}.

\begin{proof}[Proof of \cref{prop:recover-distances}]
  Let $u,v \in G$ be adjacent and fix $k \ge 1$.
  Our goal is to recover $\ell_k := \lfloor \|u-v\| + kL \rfloor$ from $\pi(G,u,v)$.
  By \cref{lem:good_paths_exist}, $\ell_k+1$ is the shortest length of a good path from $u$ to $v$ with winding number $k$.
  Since whether a path $(x_0,\dots,x_n)$ is good can be determined from $\pi(G,x_0,\dots,x_n)$, this shows that $\ell_k$ can be recovered from $\pi(G,u,v)$.
\end{proof}

\subsection{Proof of \cref{lem:recover-arc}}

Let $a,b \in G$ be adjacent.
Recall that $A_{a,b}$ is the set of points of $S$ in the shorter arc of $[a,b]$ or $[b,a]$.
As a warm-up, the reader may find it instructive to note that, when $L \ge 5$, one may easily recover $A_{a,b}$ as the intersection over $v \in S$ of all intervals $(v-2,v+2) \cap S$ which contain $\{a,b\}$.
By~\eqref{eq:floor-dist}, each such interval can be recovered as the set of vertices at graph-distance at most $2$ from $v$.

When $L \ge 3$, one may try a similar approach, recovering $A_{a,b}$ as the intersection over $v \in S$ of all intervals $(v-1,v+1) \cap S$ which contain $\{a,b\}$.
Indeed, it is not hard to check that this produces the correct set, however, as~\eqref{eq:floor-dist} does not allow to recover intervals of the form $(v-1,v+1) \cap S$, we will need to slightly modify the approach.

The proof is split into two cases according to whether $L \ge 3$ or $2<L<3$.
The following lemma will be useful in both cases.
Say that a set $U \subset G$ is \textbf{small} if it is finite and some vertex of $G$ is adjacent to every vertex in $U$. Say that $U$ is \textbf{large} if it is finite and not small, i.e., no vertex of $G$ is adjacent to every vertex in $U$.
The following simple lemma shows that a finite set is small if and only if it is contained in some open arc of length $2$.
Equivalently, a finite set is large if and only if it leaves no gap of length greater than $L-2$ in its complement.

\begin{lemma}\label{lem:union-of-intervals}
  Let $U \subset G$ be finite. 
  Then $U$ is small if and only if $U \subset (v-1,v+1)$ for some $v \in G$.
\end{lemma}

\begin{proof}
Suppose first that $U$ is small. By definition, there exists a vertex $w \in G$ that is adjacent to every vertex in $U$. Since $G$ has unit threshold, $\|u-w\|<1$ for all $u \in U$. Thus, $U \subset (w-1,w+1)$, as required.

Suppose now that $U \subset (v-1,v+1)$ for some $v \in G$. Since $G$ is \gec, there exists a vertex $w \in G$ which is adjacent to all of $U$. Thus, $U$ is small, as required.
\end{proof}

\subsubsection*{Proof of \cref{lem:recover-arc} when $L \ge 3$.}

\paragraph{Step 1.}
\emph{Let $(a,b,c,d)$ be a path in $G$ and suppose that $\{a,b,c,d\}$ is large. Then $A_{b,c}$ can be recovered from $\pi(G,a,b,c,d)$.}
\smallskip

We first show that such a path $(a,b,c,d)$ must be uni-directional.
Indeed, if the arcs $(a,b)$ and $(b,c)$ are in opposite directions in $\bbS_L$, then $\{a,b,c,d\} \subset (c-1,c+1)$, and so the set is small.
Similarly, if $(b,c)$ and $(c,d)$ are in opposite directions, then $\{a,b,c,d\} \subset (b-1,b+1)$. This shows that $\{a,b,c,d\}$ are distinct vertices and that $A_{a,b} \cap A_{b,c} = \{b\}$ and $A_{b,c} \cap A_{c,d}=\{c\}$.

For a finite $U \subset G$, we denote
\[ C(U) := \{ w \in G :  U \cup \{w\}\text{ is small} \}.\]
We will show that $A_{a,b} \cup A_{b,c}$ is precisely the set
\[ W := \{ w \in G : C(\{a,b,c\}) = C(\{a,b,c,w\}) \} .\]
Since $A_{b,c} \cup A_{c,d}$ is similarly obtained, this will show that we can determine $A_{b,c}$ as 
\[ A_{b,c} = (A_{a,b} \cup A_{b,c}) \cap (A_{b,c} \cup A_{c,d}). \]

Before showing that $A_{a,b} \cup A_{b,c} = W$, we first observe that any interval $(v-1,v+1)$ containing $\{a,b,c\}$, contains $A_{a,b} \cup A_{b,c}$ and does not contain $d$.
Indeed, no such interval contains $\{a,b,c,d\}$ since $\{a,b,c,d\}$ is large.
Suppose that $(a,b,c,d)$ winds in the positive direction (the other case being similar).
Since $L\geq 3$, the circle is a disjoint union of the arcs $[a,b),[b,c),[c,d)$, and $[d,a)$.
Thus any interval in the circle that contains $a,b,c$ but not $d$ also contains $[a,b]\cup[b,c]$.
Since $[a,b]$ is the short arc between $a,b$, we have $A_{a,b} = [a,b]\cap S$, and similarly for $A_{b,c}$.
Thus any such interval $(v-1,v+1)$ must contain $A_{a,b}\cup A_{b,c}$.

To see that $A_{a,b} \cup A_{b,c} \subset W$, fix $w \in A_{a,b} \cup A_{b,c}$ and let us show that $C(\{a,b,c\}) = C(\{a,b,c,w\})$.
The containment $C(\{a,b,c,w\}) \subset C(\{a,b,c\})$ is clear, and the opposite containment follows from the fact that any interval $(v-1,v+1)$ containing $\{a,b,c\}$ also contains $A_{a,b} \cup A_{b,c}$.

To see that $W \subset A_{a,b} \cup A_{b,c}$, let $w \in W$ and suppose towards a contradiction that $w \notin A_{a,b} \cup A_{b,c}$.
In order to reach a contradiction with the fact that $C(\{a,b,c\}) = C(\{a,b,c,w\})$, it suffices to find a vertex of $G$ which belongs to an interval $(v-1,v+1)$ containing $\{a,b,c\}$, but does not belong to any interval $(v-1,v+1)$ containing $\{a,b,c,w\}$.

Recall that $(a,b,c,d)$ is a uni-directional path, and note that one of $(a,b,c,d,w)$ or $(a,b,c,w,d)$ is also a uni-directional path. Suppose for concreteness that it is the latter, and that moreover, $(a,b,c,w,d)$ winds around the cycle in the positive direction. In particular, $A_{a,b}=[a,b] \cap S$ and $A_{b,c}=[b,c] \cap S$.
Moreover, $A := [c,w] \cap S$ is the arc between $c$ and $w$, which does not contain $a$, $b$ or $d$ (it is not necessarily the short arc between $c$ and $w$).
Note that, since any interval $(v-1,v+1)$ containing $\{a,b,c\}$ must contain $A_{a,b} \cup A_{b,c}$ and cannot contain $d$, it follows that any interval $(v-1,v+1)$ containing $\{a,b,c,w\}$ must contain $A_{a,b} \cup A_{b,c} \cup A = [a,w] \cap S$.
Observe also that any such interval is contained in $(w-2,a+2)$.
However, if $v \in (c-1,w-1)$, then the interval $(v-1,v+1)$ contains $\{a,b,c\}$, but is not contained in $(w-2,a+2)$ (note that the latter is not all of $\bbS_L$ since $(a,w)$ is longer than $(a,c)$, which in turn has length more than 1).
We have thus reached a contradiction.

\paragraph{Step 2.} 
\emph{Let $a,b \in G$ be adjacent. Then $A_{a,b}$ can be recovered from $\pi(G,a,b)$.}
\smallskip

In light of the previous step, it suffices to show that there exists a path $(x,a,b,y)$ so that $\{x,a,b,y\}$ is large.
Denote $\ell := \|a-b\|$ and suppose they are oriented so that $b=a+\ell$.
There exists $x \in (a-1,a-1+\ell/2)$ adjacent to $a$, and similarly, there exists $y \in (b+1-\ell/2,b+1)$ adjacent to~$b$.
Since $L \ge 3$, the only way $\{x,a,b,y\}$ is contained in an open arc of length $2$ is if $y-x<2$, which is not the case by our choice of $x$ and $y$.

\subsubsection*{Proof of Lemma~\ref{lem:recover-arc} when $2<L<3$.}

We write $L=2+\delta$ for some $\delta \in (0,1)$.
The reader may keep in mind that small $\delta$ is the more difficult case.

\paragraph{Step 1.}
\emph{Let $u,v \in G$. Then $\1_{\{\|u-v\|<\delta\}}$ can be recovered from $\pi(G,u,v)$.}
\smallskip

This will follow if we show that
\[ \|u-v\| < \delta \quad\iff\quad U \cup \{u\}\text{ and }U \cup \{v\}\text{ are large for some small }U .\]

Suppose first that $U \cup \{u\}$ and $U \cup \{v\}$ are large for some small $U$. Let us show that $\|u-v\|<\delta$.
Let $\{u^-,u^+\}$ be the two vertices of $U$ nearest to $u$ from either side.
Recall that a finite set is large if and only if it leaves no gap of length greater than $\delta=L-2$ in its complement. Therefore, since $U \cup \{u\}$ is large, $(u^-,u)$ and $(u,u^+)$ each has arc-length at most $\delta$. Since $U \cup \{v\}$ is large, but $U$ is small, it must be that $v \in (u^-,u^+)$ so that $\|u-v\| < \delta$, as required. 

Suppose next that $\|u-v\|<\delta$. Let us show that $U \cup \{u\}$ and $U \cup \{v\}$ are large for some small $U$.
Assume without loss of generality that $u \in (v,v+\delta)$, and let $0<\eps < (\delta - \|u-v\|)/3$. Let $V$ be the arc $(v+\delta-\eps,v-2\eps)$. Note that $v \notin V$ and $|V|=L-(\delta+\epsilon)=2-\eps<2$.
Recall that a finite set is small if and only if it is contained in some open arc of length $2$.
Thus, any finite subset of $V$ is small.
Since $(v-2\eps,v)$ and $(v,v+\delta-\eps)$ each has arc-length less than $\delta$, it follows that $U_1 \cup \{v\}$ is large for some finite $U_1 \subset V$. Similarly, since $(v-2\eps,u)$ and $(u,v+\delta-\eps)$ each has arc-length less than $\delta$, we have that $U_2 \cup \{u\}$ is large for some finite $U_2 \subset V$. Thus, $U:=U_1 \cup U_2$ is a small set such that both $U \cup \{v\}$ and $U \cup \{u\}$ are large, as required.

\paragraph{Step 2.}
\emph{Let $a,b \in G$ satisfy $\|a-b\|<\delta$. Then $A_{a,b}$ can be recovered from $\pi(G,a,b)$.}
\smallskip

By the first step, $A_v := (v-\delta,v+\delta) \cap S$ can be recovered from $\pi(G,v)$.
Let $W$ be the intersection of all $A_v$ containing $\{a,b\}$.
We claim that $A_{a,b} = W$.

To see that $A_{a,b} \subset W$, we must show that $A_v$ contains $A_{a,b}$ whenever it contains $\{a,b\}$.
Indeed, if $\{a,b\} \subset A_v$ then, since $A_v$ has arc-length $2\delta$, $\|a-b\|<\delta$ and $3\delta<L$, it follows that $A_{a,b} \subset A_v$. 

To see that $W \subset A_{a,b}$, we must show that for any $u \in G \setminus A_{a,b}$ there exists $v \in G$ such that $\{a,b\} \subset A_v$ and $u \notin A_v$.
Since $2\delta<L$, it is straightforward that such a $v$ exists.

\paragraph{Step 3.} 
\emph{Let $a,b \in G$ be adjacent. Then $A_{a,b}$ can be recovered from $\pi(G,a,b)$.}
\smallskip

Set $n := \lceil 1/\delta \rceil$, so that $1\le \delta n < 1+\delta$.
Since $\|a-b\| < 1 \le \delta n$, \gec\ implies that
there exists a uni-directional path $P= (u_0,\dots,u_n)$ in $G$ such that $u_0=a$, $u_n=b$, and $\|u_i-u_{i-1}\|<\delta$ for all $i$.
Since the long arc from $a$ to $b$ has length larger than $1+\delta$, and since the total cycle-length of $P$ is at most $\delta n < 1+\delta$, it must be that $A_{a,b} = A_{u_0,u_1} \cup \dots \cup A_{u_{n-1},u_n}$. Thus, by the previous steps, one can recover $A_{a,b}$ from $\pi(G,a,b)$.

\section{Constructing an isomorphism for rational $L$}

In this section, we prove part (2) of Theorem~\ref{thm:main}, namely, that $G_{L,S,p}$ and $G_{L,S',p'}$ are almost surely isomorphic when $L$ is rational and $S$ and $S'$ are integer-distance-free.
In light of Lemma~\ref{lem:gec}, this is an immediate consequence of the following. Let $\cG_{L,\text{idf}} \subset \cG_L$ consist of those graphs whose vertex sets are integer-distance-free.

\begin{prop}\label{prop:isomorphic}
Let $L>2$ be rational and let $G,G' \in \cG_{L,\text{idf}}$. Then $G \cong G'$.
\end{prop}

Bonato and Janssen proved the analogous statement for the case $L=\infty$ (corresponding to the real line) \cite[Theorem~3.3]{bonato2011infinite}. Our construction follows similar lines.

Let $G \in \cG_{L,S}$ and $G' \in \cG_{L,S'}$, where $S$ and $S'$ are integer-distance-free.
Our goal is to construct an isomorphism $f$ between $G$ and $G'$.
In fact, we have the flexibility to map one vertex of $G$ to an arbitrary vertex of $G'$.
Thus, by translating $S$ and $S'$ if necessary, we may assume without loss of generality that both $S$ and $S'$ contain $0$, and we aim to construct an isomorphism $f \colon S \to S'$ such that $f(0)=0$.

For a point $v \in \bbS_L$, consider the sequence $( \lfloor v + kL \rfloor )_{k \ge 0}$.
We have seen in the previous section that a similar sequence can be recovered from $\pi(G,v,0)$. Thus, when trying to construct an isomorphism between $G$ and $G'$, we shall want to preserve this sequence. That is, we will always map a vertex $v \in S$ to a vertex $v' \in S'$ in such a way that $( \lfloor v + kL \rfloor )_{k \ge 0} = ( \lfloor v' + kL \rfloor )_{k \ge 0}$.
A key observation is that, since $L=\ell/m$ is rational, $(\lfloor v + kL \rfloor)_{k \ge 0}$ is determined by its first $m$ elements.
More precisely, the sequence is periodic after subtracting the fixed sequence $(\lfloor kL \rfloor)_{k \ge 0}$.
Since there are only finitely many possibilities for the sequence, there are  many candidates for such a $v' \in S'$.

To be more precise, let
\[ L= \frac \ell m \]
be irreducible. Then there are precisely $\ell$ possibilities for the sequence $( \lfloor v + kL \rfloor )_{k \ge 0}$, according to the value of $\lfloor mv \rfloor$. We thus partition $\bbS_L$ into the arcs
\[ [\tfrac im, \tfrac{i+1}m) \qquad\text{for }0 \le i \le \ell - 1 .\]
Note that two points $u \in [\frac im, \frac{i+1}m)$ and $v \in [\frac jm, \frac{j+1}m)$ satisfy $( \lfloor u + kL \rfloor )_{k \ge 0}=( \lfloor v + kL \rfloor )_{k \ge 0}$ if and only if $i=j$.
We henceforth let $q_v \in \{0,\dots,\ell-1\}$ and $r_v \in [0,\frac 1m)$ be the unique numbers such that
\[ v = \frac{q_v}m + r_v .\]
Note that $q_v = \lfloor mv \rfloor$ and $v \in [\frac{q_v}m,\frac{q_v+1}m)$.
Moreover, since $S$ is integer-distance-free, $r_v\neq r_u$ for distinct $u,v\in S$, and in particular $r_v\neq 0$ for $v\neq 0$ in $S$, and similarly for $S'$.

Let $\bar{S} \subset S$ and $\bar{S}' \subset S'$, and suppose that both contain $0$.
A bijection $f \colon \bar{S} \to \bar{S}'$ is called an \textbf{extended step-isometry} if $f(0)=0$ and, for every $u,v \in \bar{S}$, we have
\[ q_u = q_{f(u)} \qquad\qquad\text{and}\qquad\qquad r_u < r_v \iff r_{f(u)} < r_{f(v)} .\]
Though we will not require it, we note that such an extended step-isometry $f$ satisfies that
\[ \lfloor m \cdot \|u-v\| \rfloor = \lfloor m \cdot \|f(u)-f(v)\| \rfloor \qquad\text{for all }u,v \in \bar{S} .\]

\begin{proof}[Proof of \cref{prop:isomorphic}]
  Let $S=\{s_n\}_{n \ge 0}$ and $S'=\{s'_n\}_{n \ge 0}$, where $s_0=s'_0=0$.
  Set $S_0=S'_0=\{0\}$ and let $f_0 \colon S_0 \to S'_0$ satisfy $f_0(0)=0$.
  For $n \ge 1$, we inductively define a pair of finite sets $(S_n,S'_n)$ and a bijection $f_n\colon S_n \to S'_n$ such that
  \begin{itemize}[nosep]
  \item $S_{n-1} \subset S_n \subset S$ and $S'_{n-1} \subset S'_n \subset S'$,
  \item $s_n \in S_n$ and $s'_n \in S'_n$,
  \item $f_{n-1}$ is the restriction of $f_n$ to $S_{n-1}$,
  \item $f_n$ is an extended step-isometry and an isomorphism between $G[S_n]$ and $G'[S'_n]$, where $G[U]$ denotes the subgraph of $G$ induced by $U$.
  \end{itemize}
  The limiting function $\bigcup_n f_n$ will then be the desired isomorphism.

  Fix $n \ge 0$ and suppose that we have already constructed $(S_n,S'_n)$ and $f_n$. To construct $(S_{n+1},S'_{n+1})$ and $f_{n+1}$, we use the back-and-forth method:
  First, we find a suitable image, say $s'$, for $s_{n+1}$ in $S'$ (this is the `forth' part). If the image of $s_{n+1}$ was already determined in a previous step, i.e., $s_{n+1} \in S_n$, then we skip this part and set $s'=f_n(s_{n+1})$.
  Next, we find a suitable preimage, say $t$, for $s'_{n+1}$ in $S$ (this is the `back' part). As before, if the preimage of $s'_{n+1}$ was already determined, i.e., $s'_{n+1} \in S'_n \cup \{s'\}$, then we skip this part and set $t$ to be this preimage.
  We then define $S_{n+1} = S_n \cup \{s_{n+1},t\}$, $S'_{n+1} = S'_n \cup \{s'_{n+1},s'\}$ and $f_{n+1} = f_n \cup \{ (s_{n+1},s'), (t,s'_{n+1}) \}$. In this way, the first three properties above will automatically hold, so that `suitable' refers solely to satisfying the fourth property.
  The two parts are analogous to one another, and so we only explain the first part.

  Denote $s := s_{n+1}$ and suppose that $s \notin S_n$. We wish to find a suitable image $s' \in S'$ for $s$.
  Thus we need an element of $S'$ such that $f_n \cup \{(s,s')\}$ is an extended step-isometry and an isomorphism between $G[S_n \cup \{s\}]$ and $G'[S'_n \cup \{s'\}]$. Let us first describe those candidates which ensure the former condition.

  Consider the values
  \begin{align*}
    a &:= \max \{ r_{f_n(u)} : u \in S_n\text{ and }r_u < r_s \} , \\
    b &:= \min \{ r_{f_n(u)} : u \in S_n\text{ and }r_u > r_s \} \cup \{ \tfrac 1m \} .
  \end{align*}
  Note that the set defining $a$ is not empty, since $0\in S_n$ and $0=r_0 < r_s$.
  Since $f_n$ is an extended step-isometry, any element in the set defining $a$ is strictly smaller than any element in the set defining $b$, and hence $a<b$. Denote
  \[ I := \frac{q_s}m + (a,b) .\]
  Observe that, since $S$ is integer-distance-free, $I \cap S'$ is precisely the set of $s' \in S'$ such that $f_n \cup \{(s,s')\}$ is an extended step-isometry.

  It remains only to show that $I \cap S'$ contains an element $s'$ such that $f_n \cup \{(s,s')\}$ is an isomorphism between $G[S_n \cup \{s\}]$ and $G'[S'_n \cup \{s'\}]$.
  Since $f_n$ is an isomorphism between $G[S_n]$ and $G'[S'_n]$, it suffices to show that $I \cap S'$ contains an element $s'$ which is adjacent to every element in $f_n(N(s) \cap S_n)$ and to no other element in $S'_n$.
  This will follow from the fact that $G'$ is \gec\ and has unit threshold once we show that $f_n(J)$ is contained in the open arc $(z-1,z+1)$ for some $z \in I\cap S'$, where $J=N(s) \cap S_n$.
  We will in fact show that this holds for every $z\in I$.

  Note that, since $G$ has unit threshold, $J\subset N(s) \subset (s-1,s+1)$.
  Fix $z\in I$, let $x \in J$ and denote $y:=f_n(x)$.
  We need to show that $y \in (z-1,z+1)$.
  Recall that $q_y=q_x$ and $q_z = q_s$.
  Since $x<s+1$, we have $q_x \le q_s+m$ (if $s>L-1$, so that $q_s+m \ge \ell$, this should be interpreted modulo $\ell$).
  \begin{itemize}[nosep]
  \item If $q_x \neq q_s+m$ then also $q_y < q_z+m$ which implies
    $y<z+1$.
  \item If $q_x = q_s+m$, more care is needed.
    Since $x-s = (q_x-q_s)/m + r_x-r_s < 1$ it follows that $r_x < r_s$.
    Let $u,v\in S_n$ be the points with $r_{f_n(u)}=a$ and $r_{f_n(v)}=b$.
    By the definition of $a$ and $b$, we cannot have $r_x\in(r_u,r_s)$, so $r_x\le r_u$ and so $r_y \le a$.
    Therefore,
    \[ y = \frac{q_y}{m} + r_y \le \frac{q_s+m}{m} + a < 1+z. \]
  \end{itemize}

  The other direction (that $y>z-1$) is almost identical. Either $q_x>q_s-m$ and all is well, or else $q_x=q_s-m$ and then $r_x\ge r_v$ and $r_y\ge b$, concluding as above.
  A small difference is that if $b=1/m$, there are no points in $S_n$ with $r_x>r_s$, so the latter case is impossible.
\end{proof}

For the most part, the $L=\infty$ case handled in \cite{bonato2011infinite} (corresponding to the metric space $(\R,|\cdot|)$) behaves similarly to the case of rational $L>2$.
In particular, both have the Rado property and every dense countable set which is integer-distance-free is Rado.
However, for sets which are not integer-distance-free, there are subtle differences between the two cases.
As an example, we contrast the set of rationals in $\R$ and in $\bbS_L$.

\begin{prop}
  The set of rationals $\Q$ is strongly non-Rado in $(\R,|\cdot|)$.
\end{prop}

This statement -- as well as an analogous statement for $(\R^d,\ell_\infty)$ -- appeared in the proof of Theorem~2(i) in~\cite{balister2018random}. We give a proof for completeness.

\begin{proof}
  Let $G,G' \in \cG_{\infty,\Q}$.
  Let us first show that any isomorphism $f \colon \Q \to \Q$ from $G$ to $G'$ must map $x+\Z$ to $f(x)+\Z$.
  More specifically, there exists $\eps\in\{\pm1\}$ such that $f(x+n) = f(x) + \eps n$ for all $x\in\Q$ and $n \in \Z$.
  To see this, observe that $x$ is the unique vertex in $G$ which is at graph-distance 3 in $G$ from both $x-2$ and $x+2$ (by~\eqref{eq:floor-dist}, points at distance $3$ from $x$ are precisely those in $(x-3,x-2]\cup [x+2,x+3)$).
  Moreover, if for some $u$ and $v$ there is a unique vertex $x$ with $\dist_G(x,u)=\dist_G(x,v)=3$, then necessarily $|u-v|=4$ (otherwise there would be no such vertex or infinitely many).
  This implies that $\{f(x\pm2)\} = \{f(x)\pm 2\}$.  A similar argument shows that $\{f(x\pm k)\} = \{f(x) \pm k\}$ for any integer $k\ge 2$ and any $x$.
  It is easy to deduce from this that for any $x$ there is $\eps_x\in\{\pm1\}$ such that $f(x+n) = f(x) + \eps_x n$ for all $n$.
  For any $x,y\in[0,1)$, we have by~\eqref{eq:floor-dist} that $\dist_{G'}(f(x+n),f(y+n)) = \dist_G(x+n,y+n) \le 2$ and $\dist_{G'}(f(x)+n,f(y)-n) \ge |2n|-1$.
  Thus, $\eps_x=\eps_y$ for all $x$ and $y$, which yields our claim.
  
  Now let $G,G' \in \cG_{\infty,\Q,p}$ be independent and consider two nearby points in $G$, say $0$ and $\frac12$.
  For $u,v \in \Q$, consider the sets
  \begin{align*}
    A & := \{\text{$n \in \Z : n$ and $\tfrac12+n$ are adjacent in $G$}\}, \\
    B^\pm_{u,v} & := \{\text{$n \in \Z : u\pm n$ and $v\pm n$ are adjacent in $G'$}\}.
  \end{align*}
  By \cref{lem:gec} and the above, almost surely, if $G \cong G'$ then $A=B^+_{u,v}$ or $A=B^-_{u,v}$ for some $u$ and $v$ (namely, for $u=f(0)$ and $v=f(\frac12)$, where $f$ is an isomorphism from $G$ to $G'$).
  However, since $B^{\pm}_{u,v}$ is a sequence of independent $\text{Ber}(p)$ random variables, independent also of $A$,
  this clearly has probability zero for any fixed $u$ and $v$.
  Since there are countably many choices for $u$ and $v$, we deduce that $\Pr(G \cong G')=0$.
\end{proof}

\begin{prop}
  Let $L>2$ be rational. Then $\Q \cap \bbS_L$ is Rado.
\end{prop}

\begin{proof}
  The proof is essentially the same as for integer-distance-free $S$ (\cref{prop:isomorphic}), with a small twist --
  instead of finding a suitable image for a single vertex $s$ at a time, we do so for several vertices at a time, those in a certain equivalence class of $s$.

  Let $L=\frac \ell m$ be irreducible. Say that $u,v \in \Q \cap \bbS_L$ are equivalent if $\|u-v\|$ is a multiple of $\frac1m$ and write $[u]$ for the equivalent class of $u$.
  Note that $|[u]|=\ell$ and that $[u]=v+\frac1m\{0,\dots,\ell-1\}$ for some $v \in [0,\frac 1m)$.
  We also write $[U] := \bigcup_{u \in U}[u]$.

  Let $S=\{s_n\}_{n \ge 0}$ be an enumeration of representatives of $\Q \cap \bbS_L$, where $s_0=0$ and $s_n \in (0,\frac 1m)$ for $n \ge 1$.
  The isomorphism $f \colon \Q \cap \bbS_L \to \Q \cap \bbS_L$ between $G$ and $G'$ that we aim to construct will be defined completely by a permutation of $S$ by requiring that
  \begin{equation}\label{eq:f-class}
    f(s + \tfrac im) = f(s) + \tfrac im \qquad\text{for all }s \in S\text{ and }0 \le i \le \ell -1 .
  \end{equation}
  We shall also require that $f(0)=0$.
  Thus, we only need to prescribe the value of $f$ at one element in each equivalence class, as the rest are then determined by this.

Suppose that we have already constructed a partial permutation of $S$ which gives rise through~\eqref{eq:f-class} to a function $f$ which is an extended step-isometry and an isomorphism of the induced subgraphs.
That is, $f$ is a bijection between some $S_n \subset S$ and $S'_n \subset S$, and it extends to a bijection from $[S_n]$ to $[S'_n]$ by~\eqref{eq:f-class}.
To proceed with the `forth' step of the back-and-forth method, we choose some $s \in S \setminus S_n$. We need to find an image $s' \in S \setminus S'_n$ for $s$ such that $f \cup \{(s + \frac im,s' + \frac im) : 0 \le i < \ell \}$ is an extended step-isometry and an isomorphism between $G[[S_n] \cup [s]]$ and $G'[[S'_n] \cup [s']]$.
A similar argument to the one given in the proof of \cref{prop:isomorphic} shows that there is an open interval of candidates in $(0,\frac 1m)$ which satisfy the extended step-isometry requirement. Since each such candidate satisfies the isomorphism requirement with positive (constant) probability, and since these events are independent, there almost surely exists a suitable choice for $s$.
Other than this difference, the proof proceeds identically.
\end{proof}

\begin{remark}\label{rem:non-rado}
  When $L>2$ is rational, one can also construct dense countable subsets of $\bbS_L$ which are neither Rado nor strongly non-Rado,
  so that the probability to have an isomorphism is neither zero nor one.
  One such way is to take a dense integer-distance-free set and add the points of $\{ \frac i{2m} : 0 \le i < 2\ell \}$.
  We omit the details.
  See~\cite[Theorem~2]{balister2018random} for a similar statement when the underlying metric space is a finite-dimensional normed space.
\end{remark}

\paragraph{Acknowledgements.}
This work was supported in part by NSERC of Canada.

\bibliographystyle{amsplain}
%\nocite{*}
\bibliography{library}

\end{document}